\documentclass[leqno]{article}
\usepackage{geometry}
\usepackage{graphicx}	
\usepackage[utf8]{inputenc}
\usepackage[T1]{fontenc}
\usepackage[english]{babel}
\usepackage{mathtools}
\usepackage{amsfonts,amssymb,amsmath,amsthm}

\usepackage[colorinlistoftodos]{todonotes}

\usepackage{url}
\usepackage{parskip}
\ifpdf
\usepackage{hyperref}
\else
\usepackage[dvipdfmx]{hyperref}
\fi

\usepackage{stmaryrd}

\usepackage[backend=biber,style=alphabetic]{biblatex}
\usepackage{etoolbox}
\pdfpageheight=\textheight
\pdfpagewidth=\textwidth
\hypersetup{colorlinks=true,pdfencoding=Auto}
\def\thtext#1{
  \catcode`@=11
  \gdef\@thmcountersep{. #1}
  \catcode`@=12
}

\def\threst{
  \catcode`@=11
  \gdef\@thmcountersep{.}
  \catcode`@=12
}

\theoremstyle{plain}

\newtheorem{theorem}{Theorem}[section]
\newtheorem{proposition}[theorem]{Proposition}
\newtheorem{corol}[theorem]{Corollary}

\newtheorem{lemma}[theorem]{Lemma}

\theoremstyle{definition}
\newtheorem{define}[theorem]{Definition}
\newtheorem{comm}[theorem]{Remark}

\newtheorem{example}[theorem]{Example}

 \catcode`@=11
 \def\.{.\spacefactor\@m}
 \catcode`@=12

\def\N{{\mathbb N}}

\def\R{\mathbb R}

\def\S{\Sigma}

\def\eps{\varepsilon}

\def\0{\emptyset}
\def\:{\colon}
\def\<{\langle}
\def\>{\rangle}
\def\[{\llbracket}
\def\]{\rrbracket}

\def\c{\circ}

\def\rom#1{\emph{#1}}
\def\({\rom(}
\def\){\rom)}

\def\torllim{\operatornamewithlimits{\rightleftarrows}\limits}

\def\x{\times}

\def\FF{\mathcal{F}}
\def\OO{\mathcal{O}}
\def\RR{\mathcal{R}}

\newcommand{\dGH}[1]{\operatorname{d}^{#1}_{\mathrm{GH}}}

\def\eps{{\varepsilon}}

\def\codis{\operatorname{codis}}

\def\diam{\operatorname{diam}}

\def\dis{\operatorname{dis}}

\def\GH{\operatorname{\mathcal{G\!H}}}

\def\id{\operatorname{id}}

\def\codis{\operatorname{codis}}
\def\supp{\operatorname{supp}}

\newcommand{\lowlim}[1]{\underset{#1}{\lim}}

\newcommand{\lowsup}[1]{\underset{#1}{\sup}}

\newcommand{\dist}[1]{\operatorname{d}_{#1}}
\newcommand{\di}[2]{\operatorname{d}_{#1}^{#2}}

\def\vy{\tilde{y}}

\newcommand{\va}{\tilde{a}}
\newcommand{\vx}{\tilde{x}}

\newcommand{\vb}{\tilde{b}}
\newcommand{\vA}{\tilde{A}}
\newcommand{\vR}{\tilde{R}}
\newcommand{\vB}{\tilde{B}}

\begin{document}
\title{Modifications of the Continuous Gromov--Hausdorff Distance}
\author{A.A.~Vikhrov}
\date{}
\maketitle
\tableofcontents

\begin{abstract}
This work studies modifications of the Gromov--Hausdorff distance in which the infimum is taken over pairs of morphisms of a given class $\FF$ --- the class of morphisms of a subcategory of the category of metric spaces (the $\FF$-Gromov--Hausdorff distance). We prove that the $\FF$-distance depends Lipschitz-continuously on the class $\FF$ in the Hausdorff metric on morphisms. As a consequence, for second countable $C^k$-manifolds equipped with a metric compatible with the topology, the smooth distance coincides with the continuous one: $\dGH{C_k} = \dGH{C}$. We also introduce the partially continuous and partially locally constant distances (continuity is weakened to continuity on an open subset of the support of full measure) and prove that on the class of metric spaces they coincide with the classical Gromov--Hausdorff distance.
\end{abstract}
\footnotetext{The work was supported by the Basis Foundation (grant No.~25-8-3-11-1) and the Russian Science Foundation (project No.~25-21-00152).}

\section{Introduction}

A symmetric map $d\: X \times X \to [0,\infty]$ that vanishes on the diagonal and satisfies the triangle inequality is called a \emph{generalized pseudometric}.
A generalized pseudometric that takes no infinite values is called a \emph{pseudometric}.
A generalized pseudometric that vanishes only on the diagonal is called a \emph{generalized metric}, and a pseudometric with this property is called a \emph{metric}.
The Gromov--Hausdorff distance (which we abbreviate to GH) measures how different two metric spaces are.
It was introduced independently by D.~Edwards~\cite{Edwards1975} in 1975 and by M.~Gromov~\cite{Gromov1981, Gromov1981SM} in 1981. This value is the greatest lower bound of the Hausdorff distances between the images of isometric embeddings of the two spaces into all possible ambient spaces.
An equivalent definition via the distortion of correspondences and the properties of this distance are described in detail in~\cite{BurBurIva} (see Section~\ref{sec: preliminaries} below).
The classical Gromov--Hausdorff distance does not take the topology into account: spaces that are topologically different but metrically close can be at a small or even zero distance from each other.

Usually the Gromov--Hausdorff distance is considered on the set of compact metric spaces. We consider it on the class of all metric spaces, so we work in the von~Neumann--Bernays--G\"odel axiom system, where proper classes, which generalize the notion of a set, are allowed in addition to sets.
The proper class of all metric spaces considered up to isometry is denoted by $\GH$; its structure as a proper class is described in detail in~\cite{BogatyiTuzhilinGHClass}.
We treat it as a class of representatives: exactly one space with a fixed metric is chosen from each isometry class of metric spaces.
The notion of a generalized pseudometric is naturally defined on this proper class.

Instead of arbitrary correspondences in the formula of Proposition~\ref{prop: gh via maps}, one can take pairs of maps $f\: X \to Y$ and $g\: Y \to X$ that generate the correspondence $R_{f,g}$.
The \emph{continuous Gromov--Hausdorff distance} $\dGH{C}$ is the modification in which these maps are required to be continuous.
This construction can be generalized: take a subcategory of the category of metric spaces and take the infimum over pairs of its morphisms. The resulting $\FF$-distance $\dGH{\FF}$ ($\FF$ is the class of morphisms of the subcategory) is considered in Section~\ref{sec: f-dist}.
Monotonicity in the class of morphisms (Proposition~\ref{prop:FFmonotone}) implies, in particular, that $\dGH{\FF} \geq \dGH{}$.

The continuous distance was introduced by Lim, M\'emoli and Smith~\cite{LimMemoliSmith} in their comparison of spheres: for the unit spheres $\mathbb{S}^m$ and $\mathbb{S}^n$ with $1 \leq m < n < \infty$ it turned out to be $\pi/2$, while the classical distance is strictly smaller.
Another version, which does not satisfy the triangle inequality, was proposed by Lee and Morales~\cite{LeeMorales2022} for dynamical systems and partial differential equations.
Bogatyi and Tuzhilin~\cite{TuzCont} developed the general theory of the continuous distance (in the Lim--M\'emoli--Smith version): in particular, they proved the triangle inequality and showed that this distance is intrinsic, like the classical one, but incomplete, unlike the classical one.
The generalization to set-valued maps and semicontinuous correspondences is developed in~\cite{SemenovTuzhContGH}.

The difference between the continuous and the classical distances suggests that a small extension of the class of continuous maps can make the corresponding distance equal to the classical GH-distance: for example, one can achieve this by allowing a map to have finitely many points of discontinuity.

As an example, consider the two-point simplex $\Delta_2 = \{a, b\}$ (the only non-zero distance is~1, see the precise definition in Section~\ref{sec: preliminaries}) and the segment $[0,\, 2]$.
Every continuous map $g\: [0,\, 2] \to \Delta_2$ is constant, so $\dis g = \diam[0,\, 2] = 2$, and hence $\dGH{C}(\Delta_2,\, [0,\, 2]) \geq 1$.
If we allow one point of discontinuity, then the corresponding modified distance coincides with the classical one, $\dGH{}(\Delta_2,\, [0,\, 2])$, which is equal to $1/2$ (see~\cite[Thm.~2.12]{GrigorIvanTuzSimplex} for details).
Thus $\dGH{}(\Delta_2,\, [0,\, 2]) = 1/2$, while $\dGH{C}(\Delta_2,\, [0,\, 2]) = 1$.
So for this pair one point of discontinuity is enough to recover the classical distance.
Moreover, for any $|I| \ge 2$ the vertices of the cube ${[0,\, 1]}^I$ and its skeleton have $\dGH{C, \alpha}$ strictly larger than the classical distance for $\alpha = 0$, and, if $I$ is infinite, for all $\alpha < 2^{|I|}$ (see Example~\ref{example: cube skeleton}).

We impose conditions on the maps not on the whole space, but only on a subset of full measure of the support.
First we weaken the continuity requirement in this way: we require the map to be continuous on a subset of full measure of the support that is open in the topology induced on the support.
Then we require the same for local constancy: the map must be locally constant on an open subset of the support of full measure.
This gives the \emph{partially continuous} distance $\dGH{aec}$ and the \emph{partially locally constant} distance $\dGH{aeconst}$ on the class $\GH_{mm}$ of Gromov metric triples $(X,\, d,\, \mu)$, where $\mu$ is an arbitrary boundedly finite Borel measure. The corresponding distances on the class of ordinary metric spaces are defined through them (see Section~\ref{sec: partially smooth}).
The distances $\dGH{aec}$ and $\dGH{aeconst}$ are weaker than $\dGH{C}$, and in this work we show that they coincide with the classical distance $\dGH{}$ (Theorem~\ref{theorem: aec equal original}).

This work also considers restricting the class of continuous maps to smooth ones: for manifolds with a smooth structure it is natural to define the \emph{smooth Gromov--Hausdorff distance} $\dGH{C_k}$ --- the infimum over $C^k$-smooth maps.
It turns out that the smooth structure does not give a finer distinction than the continuous one.

Putting the results together, we obtain the hierarchy
$$
  \dGH{} = \dGH{aec} = \dGH{aeconst} \leq \dGH{C} = \dGH{C_k},
$$
where on the left is the weakest distance, the classical one, which does not take the topology into account; next come the measure-controlled partially continuous and partially locally constant distances, which coincide with it; on the right are the stronger distances that take the topology and smoothness into account, the continuous and the smooth ones, which also coincide with each other.

The author is grateful to his scientific advisor Prof.~A.A.~Tuzhilin and to Prof.~A.O.~Ivanov for formulating the problem and their constant attention to the work.

\section{Preliminaries}\label{sec: preliminaries}

We first introduce the necessary notation.
The distance between points $x$ and $y$ is denoted by $|xy|$, $d(x,y)$ or $\rho(x,y)$.
The open ball with center $x$ and radius $r$ is denoted by $U_r(x)$ or $U(x,r)$, and the closed one by $B_r(x)$ or $B(x,r)$.

\begin{define}
  Let $A,B$ be non-empty subsets of a metric space $X$.
The \emph{Hausdorff distance} is the quantity
  \begin{equation*}
      \di{H}{X}(A, B) = \inf \Bigl\{r\: A \subset U_r(B)\, \&\, B \subset U_r (A)\Bigr\}.
  \end{equation*}
\end{define}

\begin{define}
  Let $A,B,X$ be metric spaces.
If $A$ is isometric to $\vA$ and $B$ is isometric to $\vB$, where $\vA$ and $\vB$ are subspaces of $X$, then the triple $(\vA, \vB, X)$ is called a \emph{realization of the pair} $(A, B)$.
\end{define}

\begin{define}\label{def: gh-realization}
  The \emph{Gromov--Hausdorff distance} between two metric spaces $A$, $B$ is the infimum of the Hausdorff distances over all realizations of the pair $(A,B)$.
In other words,
  \begin{equation*}
      \dGH{}(A,B) = \inf\bigl\{r: \text{there exists a realization } (\vA,\vB,X) \text{ of the pair $(A,B)$ with }\di{H}{X} (\vA,\vB) \leq r \bigr\}.
  \end{equation*}
\end{define}
As stated in the introduction, the class $\GH$ consists of representatives: one space with a fixed metric is chosen from each isometry class of metric spaces. The notation $X \in \GH$ means such a particular metric space, not an isometry class.
We now give another, simpler definition of the Gromov--Hausdorff distance.
\begin{define}
  Let $X$, $Y$ be metric spaces and let $\sigma \subseteq X \x Y$.
Its \emph{distortion} is defined by
  $$
  \dis \sigma = \sup \Bigl\{\bigl|\dist{X}(a,a')-\dist{Y}(b,b')\bigr|\:\, (a, b) \text{ and } (a', b') \in \sigma \Bigr\}.
  $$
  It is easy to see that if $\sigma'\subseteq\sigma$, then $\dis\sigma'\le\dis\sigma$, since the $\sup$ over a subset does not exceed the $\sup$ over the containing set.
\end{define}
The distortion of a map is the distortion of its functional correspondence:
$$
\dis f = \sup_{a, a' \in X}\bigl|\dist{X}(a,a') - \dist{Y}(f(a), f(a'))\bigr|
$$
\begin{define}
  A \emph{correspondence} between two sets $A$ and $B$ is a subset $R \subset A \x B$ such that for any $a \in A$ and $b \in B$ there exist $\va \in A$ and $\vb \in B$ with $(a,\vb)$, $(\va, b)$ in $R$.
\end{define}
We write $aRb$ to say that $a$ and $b$ are related by $R$, and denote the set of all correspondences between metric spaces $A$, $B$ by $\RR(A,B)$.
For $a \in A$ we write $R(a) = \{b \in B : (a, b) \in R\}$; similarly $R^{-1}(b) = \{a \in A : (a, b) \in R\}$ for $b \in B$.

\begin{define}\label{def: simplex}
  A metric space with $n$ points in which all non-zero distances are equal is called the \emph{$n$-point simplex} $\Delta_n$. If these distances are equal to $\lambda > 0$, we denote it by $\lambda\Delta_n$.
\end{define}

\begin{proposition}[\cite{BurBurIva}]\label{theorem:main_formula}
  For any metric spaces $A$ and $B$,
  \begin{equation*}
      2\dGH{}(A,B) = \underset{R \in \RR(A,B)}{\inf} \dis R.
  \end{equation*}
\end{proposition}

\begin{example}
  The simplest example of a correspondence $R \in \RR(X,Y)$ comes from a pair of maps $f \: X \to Y$ and $g \: Y \to X$, viewed as functional correspondences, by taking their ``union'': $R_{f,g} = f \cup g^{-1}$.
\end{example}
\begin{define}
  The \emph{codistortion} $\codis(f,g)$ is the quantity
  $$
  \codis(f,g)=\sup_{x\in X,\,y\in Y}\Bigl|\,\bigl|x\,g(y)\bigr|-\bigl|f(x)\,y\bigr|\,\Bigr|.
  $$
\end{define}
For the correspondence $R_{f,g}$ from the previous example,
$$
\dis R_{f,g} = \max\bigl\{\dis f,\,\dis g,\,\codis(f,g)\bigr\}:
$$
pairs of points in $R_{f,g}$ are of three kinds --- both taken from $f$, both from $g^{-1}$, or one from each --- and the suprema over them give exactly $\dis f$, $\dis g$ and $\codis(f,g)$.

\begin{proposition}[\cite{TuzCont}]~\label{prop: gh via maps}
  For any $X,Y\in\GH$,
  $$
  \dGH{}(X,Y) =  
   \frac12\,\inf_{\substack{f\:X\to Y\\ g\:Y\to X}}\, \dis R_{f,g} = 
  \frac12\,\inf_{\substack{f\:X\to Y\\ g\:Y\to X}}\,\sup\bigl\{\dis f,\,\dis g,\,\codis(f,g)\bigr\}.
  $$
\end{proposition}

The formula shows that the larger the class of morphisms, the smaller the infimum: if the morphisms of a subcategory $\mathcal{M}_1$ lie among the morphisms of $\mathcal{M}_2$, then $\dGH{\FF_2}\le\dGH{\FF_1}$ (formally: Proposition~\ref{prop:FFmonotone}).

\begin{proposition}[\cite{TuzCont}]~\label{prop: triangle inequality}
Let $X$, $Y$ and $Z$ be three metric spaces, and let $f: X \to Y$, $g: Y \to X$, $h: Y \to Z$, $k: Z \to Y$ be maps, denoted compactly as $X\torllim^f_gY\torllim^h_kZ$ (the upper indices are the ``forward'' maps, the lower ones the ``backward'' maps). Then $\codis(h\c f,g\c k)\le\codis(f,g)+\codis(h,k)$.
\end{proposition}
We introduce the notation $\mathbb{R}_+^n = \{(x_1, \ldots, x_n) \in \mathbb{R}^n : x_1 \geq 0\}$.

\begin{define}
    Let $M$ be a Hausdorff space with a countable base.
    A family of pairs $\{(U_i,\varphi_i)\}$, where the $U_i$ are open sets
    covering $M$ and each $\varphi_i$ is a homeomorphism of $U_i$ onto an open
    subset of $\mathbb{R}^n_+$, is called a $C^k$-atlas if
    the maps
    $$
    \varphi_i \circ \varphi_j^{-1} \bigl|_{\varphi_j(U_i \cap U_j)}
    $$
    are of class $C^k$ for all $i,j$.

    A $C^k$-structure on $M$ is a $C^k$-atlas maximal by inclusion.
    The pair $(M, C^k\textup{-structure})$ is called a $C^k$-manifold
    with boundary (possibly empty).
\end{define}
\begin{comm}
  Smoothness at the boundary is understood as follows. A function defined on a subset of the half-space $\R^n_+$ that is open in $\R^n_+$ is called $C^k$-smooth if it extends to a $C^k$-smooth function on some open subset of $\R^n$. This is consistent with the transition maps of the atlas and defines the class $C^k$ at boundary points.
\end{comm}
\begin{comm}\label{comm: paracompact}
A $C^k$-manifold (with or without boundary) is paracompact.
\end{comm}
\begin{comm}
  In this work, by a ``manifold'' we mean both ordinary manifolds and manifolds with boundary.
\end{comm}
\begin{define}
    Let $M$ and $N$ be $C^k$-smooth manifolds and let
    $f : M \to N$ be a map.

    The map $f$ is called \emph{$C^r$-smooth} ($0 \le r \le k$) if
    for each point $x \in M$ there are a chart $(U, \varphi)$ of $M$
    with $x \in U$ and a chart $(V, \psi)$ of $N$ with $f(x) \in V$, such that
    $f(U) \subset V$ and the map
    $$
    \psi \circ f \circ \varphi^{-1} :
    \varphi(U) \to \psi(V)
    $$
    is a $C^r$-smooth map between open subsets of
    $\mathbb{R}^n$ and $\mathbb{R}^m$.
\end{define}

In the classical literature on Gromov triples one usually restricts to probability or finite measures, but we work with a larger class of measures.
The reason is the following: submanifolds of $\R^n$ are often unbounded, so the measure induced on them by the Lebesgue measure is not finite (but it is boundedly finite).
\begin{define}
  An ordered triple $(X, d, \mu)$, where
  \begin{enumerate}
    \item $(X,d)$ is a metric space,
    \item $\mu$ is a Borel boundedly finite measure (all balls, or equivalently all bounded sets, have finite measure),
  \end{enumerate}
  is called a \emph{Gromov metric triple}.
  Recall that a map $F \: X \to Y$ between spaces with Borel sigma-algebras is called \emph{measurable} if the preimage $F^{-1}(B)$ of any Borel set $B \subset Y$ is Borel in $X$; for a measurable $F$, the pushforward of the measure is defined by $(F_* \mu_X)(B) = \mu_X(F^{-1}(B))$ for all Borel sets $B \subset Y$.
  Let $X = (X, \dist{X}, \mu_X)$ and $Y = (Y, \dist{Y}, \mu_Y)$ be Gromov metric triples. A map $F \: X \to Y$ is called an \emph{isomorphism} of Gromov metric triples if $F$ is an isometry of the underlying metric spaces that preserves the measure: $F_* \mu_X = \mu_Y$. An isometry is continuous and hence Borel measurable, so $F_* \mu_X$ is well defined.
  We denote by $\GH_{mm}$ the \emph{class of Gromov metric triples}, choosing one representative in each isomorphism class.
  In this work, the distance between Gromov metric triples is computed as the distance between their underlying metric spaces.
\end{define}

\section{The $\FF$-Gromov--Hausdorff distance}\label{sec: f-dist}
We now give a general construction.
\begin{define}\label{definition F distance}
  Let $\GH$ be the category whose objects are all metric spaces (one representative of each isometry class) and whose morphisms form the class $\OO$ of all maps between them.
  Let $\mathcal{M}$ be a subcategory of $\GH$ and let $\FF$ be its class of morphisms; as usual for a category, $\FF$ contains the identity maps and is closed under composition.

We write $\FF(X,Y)$ for the class of morphisms in $\FF$ acting from a metric space $X$ to $Y$, and $\FF[X,Y]$ for the union of the classes $\FF(X,Y)$ and $\FF(Y,X)$.\par

The \emph{$\FF$-Gromov--Hausdorff distance} between objects $X, Y \in \operatorname{Ob}(\mathcal{M})$ is then the following expression:
$$
\dGH{\FF}(X,Y)=\frac12\,\inf_{\substack{f\in \FF(X,Y)\\ g\in \FF(Y,X)}}\,\dis R_{f,g}=
\frac12\,\inf_{\substack{f\in \FF(X,Y)\\ g\in \FF(Y,X)}}\,\max\bigl\{\dis f,\,\dis g,\,\codis(f,g)\bigr\}.
$$

The infimum over the empty set is taken to be infinity.
\end{define}
The next example restates Proposition~\ref{prop: gh via maps}.
\begin{example}
  $\dGH{\OO} = \dGH{}$.
\end{example}
\begin{example}\label{example: continuous distance}
  If we take the subcategory with the same class of objects whose morphisms are the continuous maps, we obtain the \emph{continuous Gromov--Hausdorff distance} $\dGH{C}$.
\end{example}
\begin{proposition}
  For any subcategory $\mathcal{M}$ of $\GH$, the corresponding $\FF$-Gromov--Hausdorff distance is a generalized pseudometric on the class of objects $\operatorname{Ob}(\mathcal{M})$.
\end{proposition}
\begin{proof}
  Since the class of morphisms of a subcategory contains the identity maps, $\dGH{\FF}(X,X) = 0$ for every object $X$.
    Symmetry follows from the symmetry of the distance formula.
The triangle inequality follows from the closure of $\FF$ under composition and Proposition~\ref{prop: triangle inequality}.
\end{proof}
The next statement is a trivial consequence of the definition.
\begin{proposition}\label{prop:FFmonotone}
  Let $\mathcal{M}_1$ and $\mathcal{M}_2$ be subcategories of $\GH$ with classes of morphisms $\FF_1 \subseteq \FF_2$.
Then for any $X, Y \in \operatorname{Ob}(\mathcal{M}_1) \cap \operatorname{Ob}(\mathcal{M}_2)$,
  $$
    \dGH{\FF_2}(X,Y)\le\dGH{\FF_1}(X,Y).
  $$
\end{proposition}
\begin{corol}\label{proposition:GHleFGH}
Let $\FF$ be the class of morphisms of a subcategory $\mathcal{M}$; then for any $X, Y \in \operatorname{Ob}(\mathcal{M})$,
$$
\dGH{}(X,Y)\le\dGH{\FF}(X,Y).
$$
\end{corol}

\begin{define}
  For maps $f_1, f_2$ in the same $\OO(X,Y)$, we set
  $$
  \dist{\OO}(f_1, f_2) = \lowsup{x \in X} \dist{Y}\bigl(f_1(x), f_2(x)\bigr),
  $$
  and between maps acting between different pairs of spaces we set the distance to be infinity. This gives a generalized pseudometric on $\OO$.
\end{define}

The space $\OO$ splits into subsets consisting of the maps acting between the same pair of spaces, and the distance between two such subsets is infinity.\par
We now prove an auxiliary lemma.

\begin{lemma}\label{lemma:dis for close functions}
    Let $f_1, f_2 \in \OO(X,Y)$, $g_1, g_2 \in \OO(Y,X)$ and
    \begin{align*}
      \dist{\OO}(f_1, f_2) = C_1 \\
      \dist{\OO}(g_1, g_2) = C_2
    \end{align*}
    Then
    \begin{enumerate}
      \item $\dis(f_1) \leq \dis(f_2) + 2C_1$.
      \item $\dis(g_1) \leq \dis(g_2) + 2C_2$.
      \item $\codis(f_1, g_1) \leq \codis(f_2, g_2) + C_1 + C_2$.
    \end{enumerate}
\end{lemma}
  \begin{proof}
    We prove the first item; the second one is similar.
    \begin{align*}
      \dis(f_1) &= \lowsup{x, x' \in X} \bigl||f_1(x), f_1(x')| - |x,x'|\bigr| \leq \\
      \leq &
      \lowsup{x, x' \in X} \bigl||f_1(x), f_1(x')| - |f_2(x), f_2(x')| + |f_2(x), f_2(x')|  - |x,x'|\bigr| 
      \leq \\
      \leq&
      \lowsup{x, x' \in X} \bigl||f_1(x), f_1(x')| -  |f_2(x), f_2(x')| \bigr| + \dis(f_2)
      \leq \\
      \leq &
      \lowsup{x, x' \in X} \bigl||f_1(x), f_1(x')| - |f_1(x'), f_2(x')| + |f_1(x'), f_2(x')| -  |f_2(x), f_2(x')| \bigr| + \dis(f_2)
      \leq \\
      \leq &
      \lowsup{x, x' \in X} \bigl||f_1(x), f_2(x)| + |f_1(x'), f_2(x')| \bigr| + \dis(f_2) 
      \leq
      \dis(f_2) + 2C_1  
    \end{align*}
    Now the last item.
    \begin{align*}
      \codis(f_1, g_1) = & \lowsup{x \in X, y \in Y} \bigl||f_1(x), y| - |x, g_1(y)|\bigr| \leq \\
      \leq &
      \lowsup{x \in X, y \in Y} \bigl||f_1(x), y| - |f_2(x), y| + |f_2(x), y| - |x, g_2(y)| + |x, g_2(y)| - |x, g_1(y)|\bigr| 
      \leq \\
      \leq &
      \, C_1 + \codis(f_2, g_2) + C_2
    \end{align*}
  \end{proof}
In Theorem~\ref{theorem: hausd_dist} and its corollaries, the classes of maps need not be classes of morphisms of subcategories: the proofs use only the distance formula, and the categorical conditions (identity maps and closure under composition) are needed only for the pseudometric axioms.
\begin{theorem}~\label{theorem: hausd_dist}
  Let $\FF_1[X,Y]$ and $\FF_2[X,Y]$ be two classes of maps with
  $$
  \di{H}{\OO}(\FF_1[X,Y], \FF_2[X,Y]) = D.
  $$
  Then
  $$
    |\dGH{\FF_1}(X,Y) - \dGH{\FF_2}(X,Y)| \le D.
  $$
\end{theorem}

\begin{proof}
  Let $\eps > 0$, and let $f_1 \in \FF_1(X,Y)$ and $g_1 \in \FF_1(Y,X)$.
  Then there exist $f_2 \in \FF_2(X,Y)$ and $g_2 \in \FF_2(Y,X)$ such that $\dist{\OO}(f_1, f_2) \leq D + \eps$ and $\dist{\OO}(g_1, g_2) \leq D + \eps$.
    By Lemma~\ref{lemma:dis for close functions},
    \begin{enumerate}
      \item $\dis(f_1) \leq \dis(f_2) + 2(D + \eps)$.
      \item $\dis(g_1) \leq \dis(g_2) + 2(D + \eps)$.
      \item $\codis(f_1, g_1) \leq \codis(f_2, g_2) + 2(D + \eps)$.
    \end{enumerate}
    By the definition of the $\FF$-Gromov--Hausdorff distance,
    $$
    \dGH{\FF}(X,Y)=\frac12\,\inf_{\substack{f\in \FF(X,Y)\\ g\in \FF(Y,X)}}\,\max\bigl\{\dis f,\,\dis g,\,\codis(f,g)\bigr\},
    $$ 
    hence $\dGH{\FF_1}(X,Y) \le \dGH{\FF_2}(X,Y) + D + \eps$.
    Since $\eps$ is arbitrary and the assumption is symmetric in $\FF_1$ and $\FF_2$, we get the result.
\end{proof}

\begin{corol}~\label{corol: hausd_dist_ext}
    Let $\FF_1$ and $\FF_2$ be two classes of maps with
  $$
  \di{H}{\OO}(\FF_1, \FF_2) = D.
  $$
  Then for any pair $(X, Y)$,
  $$
    |\dGH{\FF_1}(X,Y) - \dGH{\FF_2}(X,Y)| \le D.
  $$
\end{corol}
\begin{proof}
  The case $D = \infty$ is obvious. Let $D < \infty$.\par
  Since the distance between $f_1 \: X_1 \to Y_1$ and $f_2 \: X_2 \to Y_2$ for different pairs $(X_1, Y_1)$ and $(X_2, Y_2)$ is infinity, the maps at finite distance from a map $f \in \FF_1[X,Y]$ all lie in $\FF_2[X,Y]$.
  Hence $\di{H}{\OO}(\FF_1[X,Y], \FF_2[X,Y]) \leq D$. 
  It remains to apply Theorem~\ref{theorem: hausd_dist}.
\end{proof}
\begin{corol}~\label{corol: close_function_class}
  Let $\FF_1$ and $\FF_2$ be two classes of maps.
Suppose that for given metric spaces $X$ and $Y$ the class $\FF_1[X,Y]$ is dense in $\FF_2[X,Y]$, that is, the closure of $\FF_1[X,Y]$ coincides with $\FF_2[X,Y]$.
  Then
  $$
    \dGH{\FF_1}(X,Y) = \dGH{\FF_2}(X,Y).
  $$
\end{corol}

\begin{proof}
  The density of $\FF_1[X,Y]$ in $\FF_2[X,Y]$ means that every map in $\FF_2[X,Y]$ can be approximated arbitrarily closely by maps from $\FF_1[X,Y]$; hence $\di{H}{\OO}(\FF_1[X,Y], \FF_2[X,Y]) = 0$.
  It remains to apply Theorem~\ref{theorem: hausd_dist}.
\end{proof}

\section{The smooth Gromov--Hausdorff distance}
  \begin{define}
    Let $k \in \N \cup \{0, \infty\}$ (for $k = 0$ we have $C^0 = C$).
  Consider the subcategory of $\GH$ (introduced in Definition~\ref{definition F distance})
  whose objects are $C^k$-manifolds (both with and without boundary),
  \emph{equipped with a metric that induces the topology of the manifold} (in other words,
  the identity map $(M,\tau)\to(M,\tau_d)$ is a homeomorphism; this is what it means
  for the metric to be compatible with the topology). Such a metric always exists: a manifold with a
  countable base is paracompact (see Remark~\ref{comm: paracompact}); metrizability of a
  paracompact Hausdorff space with a countable base also follows from Urysohn's
  metrization theorem~\cite{E85}. The morphisms are the $C^k$-maps between these manifolds.
We then call $\dGH{C_k}$ (Definition~\ref{definition F distance}) the $C^k$ \emph{Gromov--Hausdorff distance} between $C^k$-manifolds $M$ and $N$.
  \end{define}
  \begin{theorem}[{\cite[2.6 Theorem]{hirsch1976differential}}]\label{theorem: smooth approximation}
    Let $M$ and $N$ be manifolds of class $C^k$ with $1 \le k \le \infty$; the boundaries $\partial M$ and $\partial N$ may be non-empty.
    Then for any $r < k$ the set $C^k(M,N)$ is dense in $C^r(M,N)$ in the strong topology.
  \end{theorem}
  The book~\cite[p.~35]{hirsch1976differential} gives the following definition of the strong topology on the space of maps between manifolds.
  \begin{define}
    Let $\Phi = \{(\varphi_i, U_i)\}_{i \in A}$ be a locally finite family of charts on $M$ (each point of $M$ has a neighborhood that intersects only finitely many $U_i$), let $K = \{K_i\}_{i \in A}$ be a family of compact sets $K_i \subset U_i$, let $\Psi = \{\psi_i, V_i\}_{i \in A}$ be a family of charts on $N$, and let $\varepsilon = \{\varepsilon_i\}_{i \in A}$ be a family of positive numbers. If $f \in C^r(M,N)$ maps each $K_i$ into $V_i$, we define the \emph{strong basic neighborhood}
    $$
    \mathcal{N}^r(f; \Phi, \Psi, K, \varepsilon)
    $$
    as the set of $C^r$-maps $g \colon M \to N$ such that for all $i \in A$: $g(K_i) \subset V_i$ and
    $$
    \left\| D^\ell(\psi_i f \varphi_i^{-1})(x) - D^\ell(\psi_i g \varphi_i^{-1})(x) \right\| < \varepsilon_i, \qquad x \in \varphi_i(K_i),\ \ell = 0, \dots, r.
    $$
    All such sets form a base of the \emph{strong topology} on $C^r(M,N)$.
  \end{define}
  By Theorem~\ref{theorem: smooth approximation} (with $r=0$) the set $C^k(M,N)$ is dense in $C(M,N)$ in the strong $C^0$-topology. The uniform $\sup$-topology induced by $\dist{\OO}$ is weaker than the strong $C^0$-topology: \emph{every} $\sup$-ball is open in the strong topology (for a point $g$ in the $\sup$-ball of radius $c$ centered at $f$, the strong neighborhood $\mathcal{N}(g, \varepsilon)$ with the constant function $\varepsilon(x) \equiv \tfrac12(c - \sup_z \dist{N}(g(z), f(z)))$ lies entirely in this ball). So density in the strong topology implies density in the $\sup$-topology, and $C^k(M,N)$ is dense in $C(M,N)$ also in the $\sup$-topology. \par
  There is an equivalent definition of the $C^0$ topology between manifolds (see, for example,~\cite[p.~59]{hirsch1976differential}). We write $C_S(M,N)$ for the space $C(M,N)$ with the strong topology.
  \begin{define}
    When the manifold $M$ is paracompact and $N$ is metrizable, a base of the strong topology on $C_S(M,N)$ is formed by all sets of the form
    $$
    \mathcal{N}(f,\varepsilon) = \{ g : d(g(x), f(x)) < \varepsilon(x), \ \text{for all } x \in M \},
    $$
    where $f \in C(M,N)$ and $\varepsilon \in C(M, \mathbb{R}_+)$ are arbitrary.
  \end{define}
  Thus the set of maps $C^k(M,N)$ between the manifolds $M$ and $N$ is dense in the set $C(M,N)$.
  \begin{theorem}
    For any $C^k$-manifolds $M$, $N$ with a countable base, equipped with metrics compatible with the topology, and for all $k \in \N \cup \{0, \infty\}$, one has $\dGH{C}(M,N) = \dGH{C_k}(M,N)$.
  \end{theorem}
  \begin{proof}
    By Theorem~\ref{theorem: smooth approximation}, the $C^k$ maps from $M$ to $N$ are dense in the class of continuous maps from $M$ to $N$.
    Hence the class of $C^k$ maps between manifolds is dense in the class of continuous maps between manifolds in the $\sup$-topology induced by $\dist{\OO}$. 
    Hence the continuous and the $k$-smooth Gromov--Hausdorff distances coincide by Corollary~\ref{corol: close_function_class}.
  \end{proof}
  \begin{comm}
    The compatibility of the metric with the topology guarantees that the continuous maps between the metric spaces $(M,d_M)$, $(N,d_N)$ are the same as the topologically continuous maps between the manifolds; without this, the density of $C^k$ in $C$ would not carry over to the class of morphisms in $\dGH{C}$. The numerical value of $\dGH{C}(M,N)$ depends on the choice of the compatible metric, while the equality $\dGH{C} = \dGH{C_k}$ does not.
  \end{comm}

\section{\protect\sloppy Partially continuous and partially locally constant Gromov--Hausdorff distances}\label{sec: partially smooth}

  We first define these Gromov--Hausdorff distances between Gromov triples, and then transfer the definition to metric spaces.
  \begin{define}
    A map $f \: (X, \dist{X}, \mu_{X}) \to (Y, \dist{Y}, \mu_{Y})$ is called \emph{a.e.\ locally constant} (respectively \emph{a.e.\ continuous}) if there exists a subset $U \subset \supp(\mu_X)$ of full measure, open in the topology induced on the support $\supp(\mu_X)$, on which $f$ is locally constant (respectively continuous).
    We denote the class of such maps by $Const_{ae}$ (respectively $C_{ae}$).
  \end{define}
  The next definition repeats the formula of Definition~\ref{definition F distance} for Gromov metric triples; the classes $C_{ae}$ and $Const_{ae}$ are not classes of morphisms of a subcategory (see below), so the distances are given explicitly rather than through the categorical construction.
  \begin{define}
    On the class $\GH_{mm}$, the \emph{partially locally constant Gromov--Hausdorff distance} $\dGH{aeconst}$ is the quantity
    $$
    \dGH{aeconst}(X,Y)=\frac12\,\inf_{\substack{f\in Const_{ae}(X,Y)\\ g\in Const_{ae}(Y,X)}}\,\dis R_{f,g}.
    $$
    The \emph{partially continuous Gromov--Hausdorff distance} $\dGH{aec}$ is the quantity
    $$
    \dGH{aec}(X,Y)=\frac12\,\inf_{\substack{f\in C_{ae}(X,Y)\\ g\in C_{ae}(Y,X)}}\,\dis R_{f,g}.
    $$
    The infimum over the empty set is again taken to be infinity.
  \end{define}
  The triangle inequality does not follow directly from Definition~\ref{definition F distance}, as in the case of the general $\FF$-distance, for two reasons. First, the class $C_{ae}$ is not closed under composition: the composition of a.e.\ continuous maps can have continuum many points of discontinuity (for example, let $K \subset [0,1]$ be a Cantor set of positive Lebesgue measure, take a continuous $f$ that vanishes exactly on $K$, say $f(x) = \rho(x, K) = \inf_{k \in K} d(x,k)$, and let $g(t) = 0$ for $t = 0$ and $g(t) = 1$ for $t > 0$, discontinuous only at zero; then $g \c f$ is discontinuous at every point of the set $K$, which has positive measure), so $g \c f$ does not lie in the class $C_{ae}$.
   Second, the class $Const_{ae}$, although closed under composition (this follows from the first map being constant on each piece), does not contain all identity maps, since $\id$ is locally constant only on discrete spaces. We prove the triangle inequality from the following theorem:
  \begin{theorem}~\label{theorem: aec equal original}
    The partially locally constant Gromov--Hausdorff distance, the partially continuous Gromov--Hausdorff distance and the Gromov--Hausdorff distance coincide on the class $\GH_{mm}$. 
    In other words, $\dGH{aec} = \dGH{aeconst} = \dGH{}$.
  \end{theorem}
  \begin{proof}
    For the proof we need the following lemmas:

    \begin{lemma}[\cite{rudin1976pma}]~\label{lemma: monotonic_functions}
      Let $f\: \R \to \R$ be a monotone function.
Then it has at most countably many points of discontinuity, and all of them are non-removable discontinuities of the first kind.
    \end{lemma}

    \begin{theorem}[\cite{SEPSUPP}]~\label{lemma: separable support}
      Every boundedly finite Borel measure has a separable support.
    \end{theorem}

    \begin{lemma}\label{lemma: ae spheres}
      Let $X$ be a metric space with a boundedly finite Borel measure $\mu$.
Then for any point $x \in X$ at most countably many spheres $S(x,r)$ can have non-zero measure.
    \end{lemma}
    \begin{proof}

    Consider the function $r \to  \mu(B(x,r))$.
    It is monotone and, by Lemma~\ref{lemma: monotonic_functions}, has at most countably many points of discontinuity.
    Since $\mu(S(x,r)) = \mu(B(x,r)) - \lowlim{p \to r-} \mu(B(x,p))$, the sphere has zero measure at every point of continuity in $r$.
    All values of $r$ are such, except possibly countably many, which proves the claim.
    \end{proof}
    We now prove the main theorem.
    Since a.e.\ locally constant maps are a.e.\ continuous, $\dGH{} \leq \dGH{aec} \leq \dGH{aeconst}$.
    So it is enough to prove that the partially locally constant Gromov--Hausdorff distance coincides with the Gromov--Hausdorff distance. \par
    Let $X$ and $Y$ be two metric spaces; we write $X^s$ and $Y^s$ for the supports of the measures of $X$ and $Y$, respectively.
    To simplify the notation, we write $\mu$ for the measure on each of the spaces.
    If $\dGH{}(X,Y) = \infty$, then the equality holds ($\dGH{aeconst} \geq \dGH{}$).
    Now let $d := \dGH{}(X,Y) < \infty$.\par
    
    For an arbitrary $\eps > 0$, take at most countable $\eps$-nets $X^s_{\eps}$ and $Y^s_\eps$ in $X^s$ and $Y^s$ (they exist, since the supports $X^s$, $Y^s$ are separable by Theorem~\ref{lemma: separable support}). 
    Set $X^{ns} = X \backslash \supp \mu$, $Y^{ns} = Y \backslash \supp \mu$, $X_\eps = X^{ns} \cup X^s_{\eps}$,  $Y_\eps = Y^{ns} \cup Y^s_{\eps}$.
    Since $X_\eps$ and $Y_{\eps}$ are $\eps$-nets, the triangle inequality gives
    $$
    \dGH{}(X_{\eps},Y_{\eps}) \leq  \dGH{}(X_{\eps},X) + \dGH{}(X,Y) + \dGH{}(Y,Y_{\eps}) \leq d + 2\eps.
    $$
    By Proposition~\ref{theorem:main_formula}, there exists a correspondence $R \in \RR(X_{\eps}, Y_{\eps})$ with $\dis(R) \leq 2d + 5\eps$.
    Number the elements of the nets: $X^s_\eps = \{x_i\}$, $Y^s_\eps = \{y_i\}$.
    Consider the balls $B^X_{r}(x_i)$, $B^Y_{r}(y_i)$, and also $B^{X^s}_{r}(x_i) = B^X_{r}(x_i) \cap X^s$ and $B^{Y^s}_{r}(y_i) = B^Y_{r}(y_i) \cap Y^s$.
    The closed balls (of the spaces $X^s$ and $Y^s$) $B^{X^s}_{r}(x_i)$ and $B^{Y^s}_{r}(y_i)$ cover $X^s$ and $Y^s$, respectively, for $r \geq \eps$. 

    We build a new correspondence $\vR$: 
    $$
    \vR = \cup_{(x,y)\in R} B^X_{2\eps}(x) \x B^Y_{2\eps}(y).
    $$
    It is indeed a correspondence, since the balls $B^X_{2\eps}(x)$ and $B^Y_{2\eps}(y)$ cover the whole spaces $X$ and $Y$.\par

    For every pair $(\vx,\vy) \in \vR \backslash R$ there exists a pair $(x,y) \in R$ with $\max(|x \vx|, |y \vy|) \leq 2\eps$.
    Each of the four points of the two pairs being compared moves by at most $2\eps$ when passing from $R$ to $\vR$, so the difference of the distances grows by at most $8\eps$, that is, $\dis(\vR) \leq \dis(R) + 8\eps \leq 2\dGH{}(X,Y) + 13\eps$. 
    Then for every $\vR_c \subseteq \vR$ we have $\dis(\vR_c) \leq 2\dGH{}(X,Y) + 13\eps$.
    We need to build maps $f \: X \to Y$, $g \: Y \to X$ that are continuous on an open subset of the support of full measure and such that the correspondence $R_{f, g}$ is a subcorrespondence of $\vR$; then the theorem will follow.\par

    We build the map $f$; the map $g$ is built in the same way.
    For each $i$ there exists $r_i$ such that $\eps < r_i < 2\eps$ and $\mu(S(x_i,r_i)) = 0$: by Lemma~\ref{lemma: ae spheres} the ``bad'' radii form an at most countable set, while the interval $(\eps, 2\eps)$ is uncountable.
    We build the sequence of sets $\vB^X_i = B^{X^s}(x_i, r_i) \backslash \underset{j < i}{\cup}U^{X^s}(x_j, r_j)$.
    These are closed sets, and the measure of their boundary is zero (since it lies in an at most countable union of boundaries of zero measure).
    Since the balls $B^{X^s}(x_i, r_i)$ cover $X^s$, the union of the sets $\vB^X_i$ also coincides with $X^s$.
    The set $U := \bigcup_i \bigl(\operatorname{Int}_{X^s}(\vB^X_i) \setminus \bigcup_{j < i} \vB^X_j\bigr)$ is open in the induced topology of the support $X^s$ (each term is the difference of an open set and a finite union of closed sets, hence open) and has full measure, since $X^s \setminus U \subset \bigcup_i \partial_{X^s}(\vB^X_i)$ and the measure of the boundaries is zero.\par
    Define the map $f$ on $X^s$: choose one element $y_i$ from each set $R(x_i)$ and set $f(x) = y_{i(x)}$, where $i(x)$ is the smallest index $i$ with $x \in \vB^X_i$ (the union of the sets $\vB^X_i$ coincides with $X^s$, so $i(x)$ is defined for all points). On each set $\operatorname{Int}_{X^s}(\vB^X_i) \setminus \bigcup_{j < i}\vB^X_j$ the map $f$ is constant, hence continuous.
    For a point $x \in X^{ns} = X \setminus X^s$ set $f(x)$ equal to an arbitrary element of $R(x)$ (this set is non-empty, since $R$ is a correspondence on $X_\eps \supset X^{ns}$).

    Thus we have built maps $f$ and $g$, continuous on open subsets of the supports of full measure, with the following property: $\dis R_{f,g} \leq 2\dGH{}(X,Y) + 13 \eps$. Since $\eps$ is arbitrary, we get the result.
  \end{proof}
  \begin{comm}
    In fact, the maps between the sets $\vB^X_i$ can be arbitrary, not necessarily constant.
  \end{comm}
  
  \begin{proposition}~\label{proposition: no depend of measure}
    The partially locally constant Gromov--Hausdorff distance and the partially continuous Gromov--Hausdorff distance are generalized pseudometrics on $\GH_{mm}$.
    Moreover, the value of the partially continuous or partially locally constant Gromov--Hausdorff distance between triples $(X, \dist{X}, \mu_{X})$ and $(Y, \dist{Y}, \mu_{Y})$ does not depend on the measures $\mu_{X}$ and $\mu_{Y}$.
  \end{proposition}
  \begin{proof}
    By Theorem~\ref{theorem: aec equal original}, both distances coincide with the classical Gromov--Hausdorff distance computed for the metric spaces underlying the triples, which is a generalized pseudometric independent of the choice of the measures.
  \end{proof}
  Note that every metric space admits a boundedly finite Borel measure, for example a Dirac measure or a countable linear combination of Dirac measures whose support intersects every ball in a finite set.
  \begin{define}
    The partially locally constant (resp.\ partially continuous) Gromov--Hausdorff distance between \emph{metric spaces $(X, \dist{X})$ and $(Y, \dist{Y})$} is the partially locally constant (resp.\ partially continuous) distance between the metric triples $(X, \dist{X}, \mu_{X})$ and $(Y, \dist{Y}, \mu_{Y})$ for arbitrary boundedly finite Borel measures $\mu_{X}$ and $\mu_{Y}$.
  \end{define}
  By Proposition~\ref{proposition: no depend of measure}, this definition is well defined. Definition~\ref{def: dGH C n} and its examples are placed in Appendix~\ref{sec: appendix}.

\appendix
\section{Gromov--Hausdorff distances with discontinuities}\label{sec: appendix}

  In the introduction we saw that for the pair $\Delta_2$, $[0,\, 2]$ allowing one point of discontinuity is enough for the difference between the continuous and the classical distances to disappear.
  In general this is not so: in the example below, for the vertices of the cube ${[0,\, 1]}^I$ and its skeleton the distance $\dGH{C,\alpha}$ is equal to $1$ for all $\alpha < 2^{|I|-2}$, while the classical one does not exceed $\frac12$.
  In particular, for finite $|I| \ge 3$ the distances differ even for $\alpha = 1$, and for infinite $I$ they differ for all $\alpha < 2^{|I|}$.

  \begin{define}\label{def: dGH C n}
    Let $\alpha$ be a cardinal.
    A map is called \emph{$\alpha$-discontinuous} if the set of its points of discontinuity has cardinality at most $\alpha$.
    The \emph{Gromov--Hausdorff distance with $\alpha$ discontinuities} is the quantity
    $$
    \dGH{C,\alpha}(X,Y)=\frac12\,\inf \dis R_{f,g},
    $$
    where the infimum is taken over all pairs of $\alpha$-discontinuous maps $f \: X \to Y$ and $g \: Y \to X$.
    For $\alpha = 0$ we get the continuous distance $\dGH{C}$.
  \end{define}

  \begin{example}\label{example: cube skeleton}
    Let $I$ be an arbitrary set of cardinality at least two.
    Let $\S_I$ denote the skeleton of the cube ${[0,\, 1]}^I$ with the intrinsic metric $\tilde d$, and let $V_I = {\{0, 1\}}^I$ be the set of its vertices; set $d = \min(\tilde d,\, 2)$ and consider both spaces with the metric $d$.
    Then $\dGH{}(V_I, \S_I) \le \frac12$, $\dGH{C,\alpha}(V_I, \S_I) = 1$ for all $\alpha < 2^{|I|-2}$, and $\dGH{C,\alpha}(V_I, \S_I) \le \frac12$ for all $\alpha$ at least the number of edges of the skeleton.
  \end{example}
  \begin{proof}
    First of all, $d$ is a metric: for $a, b \ge 0$ one has $\min(a + b,\, 2) \le \min(a,\, 2) + \min(b,\, 2)$, so the triangle inequality is preserved, and the other axioms are obvious.
    The distance $\tilde d$ between vertices is the number of coordinates in which they differ: a path along edges changes the coordinates one by one.
    Hence distinct vertices are at distance at least $1$ from each other, so the connected subsets of $V_I$ are one-point sets.
    Vertices that differ in exactly two coordinates are at distance $2$ from each other; together with $d \le 2$ this gives $\diam V_I = \diam \S_I = 2$.
    A two-dimensional face of the cube is a square in which all coordinates except two are fixed by the values $0$ or $1$; the boundary of such a face, its four edges, is connected and contains opposite vertices at distance $2$.

    \emph{Lemma.}
    Let a map $\phi \: \S_I \to V_I$ be continuous on a connected set $C \subseteq \S_I$ containing vertices $v, w$ with $d(v, w) = 2$.
    Then $\phi|_C$ is constant, and every correspondence containing the graph of $\phi$ has distortion $2$.
    Indeed, the restriction $\phi|_C$ is continuous, its image is a connected subset of $V_I$, hence a one-point set; therefore the pairs $(\phi(v), v)$ and $(\phi(w), w)$ contribute the distortion $d(v, w) = 2$, and the distortion cannot be larger than $2$, since distances in both spaces do not exceed $2$.

    We now turn to the estimates.
    In (i) we prove the upper bound for the modified distance for any $\alpha$ at least the number of edges of the skeleton, and derive from it the bound for the classical one: the classical distance does not exceed the modified one.
    In (ii) we prove the lower bound for the modified distance: it follows from the lemma applied to a face whose boundary contains no points of discontinuity.

    (i) We prove the upper bounds.
    Take $f$ to be the inclusion $V_I \subseteq \S_I$, and take $g$ to be the map $h \: \S_I \to V_I$ sending each point of an edge to the nearest vertex of this edge (at the midpoint of an edge, to either of the two).
    In each pair $(x, y)$ of the correspondence $R_{f,g}$ the distance $d(x, y)$ does not exceed $\frac12$, so by the triangle inequality $\dis R_{f,g} \le 1$.
    The map $h$ is discontinuous exactly at the midpoints of the edges, one per edge, and $f$ is continuous, so for any $\alpha$ at least the number of edges of the skeleton, $\dGH{C,\alpha}(V_I, \S_I) \le \frac12$.
    The classical distance does not exceed the modified one, hence $\dGH{}(V_I, \S_I) \le \frac12$.

    (ii) Let $\alpha < 2^{|I|-2}$ and let $g \: \S_I \to V_I$ be continuous outside a set $F$ with $|F| \le \alpha$.
    We prove that every correspondence containing the graph of $g$ has distortion $2$ by applying the lemma to the boundary of a suitable two-dimensional face.
    Take any distinct coordinates $i, j \in I$ and consider the two-dimensional faces with free coordinates $i$ and $j$: for each $\sigma \in {\{0, 1\}}^{I \setminus \{i, j\}}$ set
    $$
    Q_\sigma = \bigl\{x \in {[0,\, 1]}^I : x_k = \sigma_k \text{ for all } k \in I \setminus \{i, j\}\bigr\}.
    $$
    There are exactly $2^{|I|-2}$ such faces (for infinite $I$ this is $2^{|I|}$); the boundary $\partial Q_\sigma$ is the four edges of the square $Q_\sigma$.
    All these faces are parallel --- they have the same free coordinates --- and are pairwise disjoint.
    In particular, each point of discontinuity belongs to at most one face, so the points of discontinuity lie on the boundaries of at most $|F| \le \alpha$ faces; the total number of faces is $2^{|I|-2} > \alpha$, so there is a face $Q$ with $\partial Q \cap F = \emptyset$.
    The boundary of $Q$ is connected and contains opposite vertices at distance $2$; the map $g$ is continuous at every point of it, and by the lemma every correspondence containing the graph of $g$ has distortion $2$, whence $\dGH{C,\alpha}(V_I, \S_I) = 1$.
  \end{proof}

\printbibliography

\end{document}